\documentclass[11pt]{amsart}

\usepackage[margin=1in]{geometry}
\usepackage{amsmath,amssymb,amsthm,mathtools}
\usepackage[colorlinks=true,linkcolor=blue,citecolor=blue,urlcolor=blue,
  pdftitle={Zero-Sum Cycles in Regular Digraphs},
  pdfauthor={Varun Sivashankar}]{hyperref}

\AtEndEnvironment{thebibliography}{\enlargethispage{4\baselineskip}}

\makeatletter
\def\section{\@startsection{section}{1}%
  \z@{1.5\linespacing\@plus.6\linespacing}{.65\linespacing}%
  {\normalfont\Large\bfseries\raggedright}}
\def\subsection{\@startsection{subsection}{2}%
  \z@{1.05\linespacing\@plus.5\linespacing}{.4\linespacing}%
  {\normalfont\large\bfseries\raggedright}}
\makeatother

\newtheorem{theorem}{Theorem}[section]
\newtheorem{proposition}[theorem]{Proposition}
\newtheorem{lemma}[theorem]{Lemma}
\newtheorem{corollary}[theorem]{Corollary}
\newtheorem{conjecture}[theorem]{Conjecture}
\newtheorem{question}[theorem]{Question}
\theoremstyle{remark}
\newtheorem*{remark}{Remark}

\newcommand{\per}{\operatorname{per}}
\newcommand{\tr}{\operatorname{tr}}
\newcommand{\diag}{\operatorname{diag}}

\title{Zero-Sum Cycles in Regular Digraphs}
\author{Varun Sivashankar}
\address{Department of Mathematics, Princeton University,
Princeton, New Jersey, USA}
\email{varunsiva@princeton.edu}
\date{}

\begin{document}

\begin{abstract}
Let $\Gamma$ be a finite group of order $k\ge2$, and label the edges of a simple loopless $d$-regular digraph $D$ by elements of $\Gamma$. A directed cycle is zero-sum if the ordered product of its labels is the identity of $\Gamma$. We prove that a zero-sum cycle exists whenever $d\ge e^3(k-1)$. We also prove that every labelled $d$-regular digraph contains $\Omega(d/k)$ pairwise vertex-disjoint zero-sum cycles. When $d\ge50k$, it contains $\Omega(d^2/k)$ pairwise edge-disjoint zero-sum cycles. All three results are asymptotically optimal. The existence and packing results extend to Eulerian digraphs whose minimum and maximum common degrees $\delta$ and $\Delta$ satisfy $\delta^3/\Delta^2=\Omega(k)$. The techniques extend a determinant--permanent argument of Friedland for even directed cycles.
\end{abstract}

\maketitle

\section{Introduction}

Throughout, digraphs are finite, simple, and loopless.  Let $D$ be a
nonempty digraph.  Thus $D$ has at most one edge from $u$ to $v$ for
each ordered pair $(u,v)$, there are no edges from a vertex to itself,
and both opposite edges $uv$ and
$vu$ may be present.  For a vertex $v$, let $d^+(v)$ and $d^-(v)$ denote
its outdegree and indegree, respectively.  We call $D$ \emph{$d$-regular}
if $d^+(v)=d^-(v)=d$ for every vertex $v$, and \emph{Eulerian} if
$d^+(v)=d^-(v)$ for every vertex $v$.  For an Eulerian digraph, write
$d(v):=d^+(v)=d^-(v)$, $\delta:=\min_v d(v)$, and
$\Delta:=\max_v d(v)$.  A directed cycle is a sequence
\[
  C=v_1v_2\cdots v_\ell v_1
\]
whose vertices $v_1,\ldots,v_\ell$ are distinct and whose displayed
edges belong to $D$.  In particular, cycles are simple rather than
arbitrary closed directed walks.  We regard two cycles that differ only
by a cyclic rotation as the same.

Let $\Gamma$ be a finite group, written multiplicatively, with identity
$1_\Gamma$, and let $w:E(D)\to\Gamma$ be an edge-labelling.  The ordered
label product of $C$ is
\[
  w(C):=w(v_1v_2)w(v_2v_3)\cdots w(v_\ell v_1).
\]
Changing the initial vertex conjugates this product, so the condition
$w(C)=1_\Gamma$ does not depend on the choice of initial vertex.  We call
such a cycle a \emph{zero-sum cycle}.  The most familiar case is
$\Gamma=\mathbb Z_k$ with every edge labelled $1$.  Then the zero-sum
cycles are exactly those whose lengths are divisible by $k$.

A natural question, going back to Friedland~\cite{Friedland} and
Alon--Linial~\cite{AlonLinial}, is the following: how large must $d$ be,
in terms of $k=|\Gamma|$, to force a zero-sum cycle under every
$\Gamma$-labelling?

\subsection{Main results}

Our main results are as follows.
\begin{itemize}
\item Theorem~\ref{thm:main} states that every $d$-regular digraph with
a $\Gamma$-labelling contains a zero-sum directed cycle whenever
$d\ge e^3(k-1)$.

\item Theorems~\ref{thm:vertex-packing} and~\ref{thm:packing} give the
regular packing bounds.  For every $d\ge1$, every $d$-regular digraph
with a $\Gamma$-labelling contains at least
$\lfloor d/(25k)\rfloor$ pairwise vertex-disjoint zero-sum directed
cycles.  If $d\ge50k$, it also contains at least $d^2/(100k)$ pairwise
edge-disjoint zero-sum directed cycles.

\item We generalize the existence result to Eulerian digraphs with
unequal degrees.  Theorem~\ref{thm:eulerian} gives a zero-sum directed
cycle whenever $\Delta>0$ and $\delta^3/\Delta^2\ge e^3(k-1)$.

\item Theorems~\ref{thm:eulerian-vertex-packing}
and~\ref{thm:eulerian-edge-packing} give absolute constants $c,C>0$
such that every Eulerian digraph with a $\Gamma$-labelling and $\Delta>0$
contains at least $\lfloor c\delta^3/(k\Delta^2)\rfloor$ pairwise
vertex-disjoint zero-sum directed cycles.  Whenever
$\delta^3/\Delta^2\ge Ck$, it also contains at least
$c\delta^4/(k\Delta^2)$ pairwise edge-disjoint zero-sum directed
cycles.

\item All the results above apply to $\Gamma=\mathbb Z_k$ with every
edge labelled $1$, giving cycles whose lengths are divisible by $k$.
Corollary~\ref{cor:all-moduli} also shows that every $d$-regular digraph
with $d\ge e^3(k-1)$ contains pairwise edge-disjoint cycles
$C_2,\ldots,C_k$ such that $r$ divides $|C_r|$ for every $2\le r\le k$.
\end{itemize}

When $\delta=\Delta=d$, Theorem~\ref{thm:eulerian} reduces to
Theorem~\ref{thm:main}, and the Eulerian packing bounds have orders
$d/k$ and $d^2/k$.  We give the regular proofs separately because they
are much shorter.  The weighted arguments
for unequal degrees are given in Section~5.

The linear dependence on $k$ in Theorems~\ref{thm:main}
and~\ref{thm:eulerian} is necessary: Proposition~\ref{prop:circulant-obstruction}
gives a strongly connected $k$-regular digraph with a
$\mathbb Z_k$-labelling and no zero-sum cycle.  The dependence on $d$
and $k$ in Theorems~\ref{thm:vertex-packing} and~\ref{thm:packing} is
also optimal up to constants.
Proposition~\ref{prop:obstruction} gives at most $d/k$ pairwise
vertex-disjoint zero-sum cycles and at most $d^2/k$ pairwise
edge-disjoint zero-sum cycles.

\subsection{Earlier results and proof strategy}

Our proof extends the following argument of Friedland for even directed
cycles.  Since a cycle of length $\ell$ has sign $(-1)^{\ell-1}$, the
disjoint-cycle decomposition of $\sigma$ gives
\[
  \operatorname{sgn}(\sigma)
  =\prod_{C\text{ a cycle of }\sigma}(-1)^{|C|-1}.
\]
Let $A$ be the adjacency matrix of a $d$-regular digraph on
$n$ vertices and put $M=I+A$.  Writing $S_n$ for the set of
permutations of $\{1,\ldots,n\}$, by definition,
\[
  \per M=\sum_{\sigma\in S_n}\prod_{v=1}^n M_{v,\sigma(v)},
  \qquad
  \det M=\sum_{\sigma\in S_n}\operatorname{sgn}(\sigma)
             \prod_{v=1}^n M_{v,\sigma(v)}.
\]
A product in these sums is nonzero precisely when the cycles of
$\sigma$ form a directed cycle cover after we add a loop at every
vertex of $D$.  The loops are the fixed points of
$\sigma$, and every nontrivial cycle of $\sigma$ is a directed cycle in
$D$.  Hence, if $D$ has no even directed cycle, every cycle in the
decomposition of $\sigma$ has positive sign, and therefore
$\det M=\per M$.  Since every row and column sum of $M$ is $d+1$, the
matrix $M/(d+1)$ is doubly stochastic.
The van der Waerden theorem and Hadamard's inequality therefore give
\[
  \per M\ge(d+1)^n\frac{n!}{n^n}
  >\left(\frac{d+1}{e}\right)^n,
  \qquad
  |\det M|\le(d+1)^{n/2}.
\]
These bounds are incompatible for every $d\ge7$, because
$\sqrt{d+1}>e$.  Thus every $d$-regular digraph with $d\ge7$ contains
an even directed cycle~\cite{Friedland}.

Vazirani and Yannakakis~\cite{VaziraniYannakakis} proved the underlying
characterization: for a digraph with a loop at every vertex, the
adjacency matrix has equal determinant and permanent if and only if the
digraph has no even directed cycle.

Alon and Bregman used the same determinant-permanent framework to prove
that every $8$-uniform $8$-regular hypergraph is
$2$-colorable~\cite{AlonBregman}.  Thomassen later sharpened this to
every $r$-uniform
$r$-regular hypergraph with $r\ge4$~\cite{Thomassen1992}.

Alon and Linial pointed out why Friedland's argument does not immediately
extend to cycles of length $0$ modulo $k$ for
$k>2$~\cite{AlonLinial}.  Signs distinguish even cycles from odd cycles.
For
a group-labelled cycle, however, one must distinguish the identity
product from every nonidentity product, and ordinary determinant signs
provide no such distinction.  Their proof instead uses a random
vertex colouring and the Lov\'asz local lemma.  The same proof gives the
following result.

\begin{proposition}[Alon--Linial, adapted]\label{prop:alon-linial}
Let $\Gamma$ be a finite group of order $k\ge2$, and let $D$ be a
$\Gamma$-labelled digraph.  Put
\[
  \delta^+:=\min_v d^+(v),
  \qquad
  \Delta^-:=\max_v d^-(v).
\]
If
\[
  e(\Delta^-\delta^++1)
  \left(1-\frac1k\right)^{\delta^+}<1,
\]
then $D$ contains a zero-sum directed cycle.
\end{proposition}

For an Eulerian digraph with $\delta\le d(v)\le\Delta$, the two
sufficient conditions for a zero-sum cycle are
\[
\begin{array}{rl}
\text{Theorem~\ref{thm:eulerian}:}
  &\displaystyle k\le
    1+\frac{\delta}{e^3(\Delta/\delta)^2},\\[1.2ex]
\text{Alon--Linial:}
  &\displaystyle k<
    \frac{\delta}{1+\log(\delta\Delta+1)}.
\end{array}
\]
The second follows from Proposition~\ref{prop:alon-linial} and
$(1-1/k)^\delta<e^{-\delta/k}$.
When $\Delta/\delta$ is bounded, our theorem allows $k$ of order
$\delta$, rather than $\delta/\log(\delta\Delta)$.  In particular, it
removes the logarithmic loss in the regular case.  More generally, up
to the distinction between $k$ and $k-1$, our theorem is stronger while
$\Delta/\delta$ is below
$e^{-3/2}\sqrt{1+\log(\delta\Delta+1)}$.  For larger degree ratios the
Alon--Linial bound is stronger, and it also applies without the Eulerian
assumption.

Alon~\cite[Proposition~4.1]{AlonDisjoint} observed that combining the
vertex-partitioning argument of Alon, McDiarmid, and
Molloy~\cite[Lemmas~1 and~2]{AlonMcDiarmidMolloy} with the Alon--Linial
theorem gives $\Omega_k(d^2)$ edge-disjoint cycles whose lengths are
divisible by $k$ in every $d$-regular digraph.  Tracking the dependence
on $k$ in the same argument and applying
Proposition~\ref{prop:alon-linial} inside each colour class gives
absolute constants $c_0,C_0>0$ such that $d\ge C_0k\log k$ forces
$c_0d/(k\log k)$ pairwise vertex-disjoint zero-sum cycles and
$c_0d^2/(k\log k)$ pairwise edge-disjoint zero-sum cycles.  Thus
Theorems~\ref{thm:vertex-packing} and~\ref{thm:packing} remove the factor
$\log k$ from both the degree threshold and the packing guarantees.

We extend Friedland's determinant-permanent argument.  We assign a
matrix to each group element and use these matrices as the blocks of a
labelled adjacency matrix.  The resulting determinant distinguishes
cycles whose label product is the identity from those whose label
product is not.  We can then compare a permanent lower bound with a
determinant upper bound as Friedland did.  MacMahon's master
theorem~\cite{MacMahon,Chu} and Kassel and L\'evy's block-determinant
identity~\cite[Theorem~3.1]{KasselLevy} motivate the construction and
the determinant calculation.

For the packing theorems, we apply the same identity to a polynomial
over directed cycle covers.  It shows that some cycle cover contains
order $d/k$ zero-sum cycles, which are pairwise vertex-disjoint.
Deleting the entire cover preserves regularity, so repeating the
argument gives order $d^2/k$ pairwise edge-disjoint zero-sum cycles.
For Eulerian digraphs, we add diagonal weights to allow uncovered
vertices and use the stronger Bethe permanent lower bound.

\subsection{Related work}

Zero-sum problems in combinatorics go back at least to the theorem of
Erd\H{o}s, Ginzburg, and Ziv~\cite{ErdosGinzburgZiv}.  Caro~\cite{Caro}
surveys the area.  For directed cycles, Thomassen~\cite{Thomassen1985}
proved that minimum outdegree $\lfloor\log_2 n\rfloor+1$ forces an even
directed cycle, and constructed examples with minimum outdegree
$\lfloor\frac12\log_2 n\rfloor$ and no such cycle.  Gutin, Sudakov, and
Yeo~\cite{GutinSudakovYeo} modified the construction to give $n$-vertex
digraphs with no even directed cycle in which every vertex has indegree
and outdegree at least $c\log n$ for an absolute constant $c>0$.  Thus
bounds on the minimum indegree and outdegree alone do not force an even
directed cycle.  Under strong
connectivity, Thomassen
proved that minimum indegree and outdegree at least $3$ force a directed
cycle of even total weight under every $\{0,1\}$ edge
weighting~\cite{Thomassen1992}.  Every weak component of a regular digraph is
strongly connected, so this implies the existence of a zero-sum cycle
in every $d$-regular $\mathbb Z_2$-labelled digraph with $d\ge3$.
Diwan~\cite{Diwan} asked for an analogue for $\mathbb Z_k$-weighted
strongly connected digraphs.

The constant $\mathbb Z_k$-labelling by $1$ connects our problem to the
extensive literature on cycles of length modulo
$k$ in undirected graphs~\cite{Erdos1976,BollobasModulo,ThomassenModulo,SudakovVerstraete,BaiGrzesikLiProrok}.  A central minimum-degree question is Dean's conjecture that minimum degree at least $k$ forces a cycle whose length
is divisible by $k$~\cite{Dean1988}.  The conjecture is known for
$k=3,4$ and for every $k\ge6$, leaving only
$k=5$~\cite{ChenSaito,DeanLesniakSaito,LuoMaZhao}.  Related results treat
weaker degree hypotheses, other residue classes, and stronger connectivity
assumptions~\cite{GaoHuoLiuMa,BaiGrzesikLiProrok}.

For arbitrary group labellings in undirected graphs, Diwan~\cite{Diwan} proved that if the vertices and edges of an
undirected graph are weighted by a finite abelian group $\Gamma$, then
minimum degree at least $2|\Gamma|-1$ forces a zero-sum cycle.  He
conjectured that $|\Gamma|+1$ suffices and constructed examples of
minimum degree $|\Gamma|$ with no zero-sum cycle.  This result allows
arbitrary host graphs and also vertex weights, but assumes that
$\Gamma$ is abelian.  Our results instead concern zero-sum cycles in
directed graphs with arbitrary (possibly nonabelian) edge labels.

Alon, McDiarmid, and Molloy~\cite{AlonMcDiarmidMolloy} studied cycle
packings in regular digraphs, proving a quadratic lower bound for
pairwise edge-disjoint directed cycles and conjecturing the sharp bound
$\binom{d+1}{2}$.  Alon~\cite{AlonDisjoint} proved linear and quadratic
lower bounds for vertex-disjoint and edge-disjoint directed cycles,
respectively, under a minimum-outdegree assumption.  These cycles have
no restriction on their labels.  The consequences for zero-sum
packings were described in Section~1.2.

A complementary line of work assumes that the underlying digraph is complete.
For a nontrivial finite group $\Gamma$, let $n(\Gamma)$ be the least
$N$ such that every $\Gamma$-labelling of the complete bidirected graph
on $N$ vertices has a zero-sum cycle.  Alon and
Krivelevich~\cite{AlonKrivelevich} introduced this parameter for cyclic
groups.  The standard ordering construction gives
$n(\mathbb Z_k)\ge k+1$~\cite{MeszarosSteiner,Diwan}.  Subsequent work
proved $n(\Gamma)\le8|\Gamma|$ for finite
abelian groups~\cite{MeszarosSteiner}, then
$n(\Gamma)\le2|\Gamma|-1$ for every nontrivial finite
group~\cite{BerendsohnBoyadzhiyskaKozma,AkramiEtAl}.  Campbell et al.
proved that $n(\Gamma)\le|\Gamma|+1$ for odd-order groups and state that
a proof of the even-order case will appear in a
sequel~\cite{CampbellEtAl}.  Sharper bounds for
elementary abelian groups
appear in~\cite{LetzterMorrison,ChristophEtAl}.  All these results
crucially use the fact that the digraph is complete.  Our results apply
to $d$-regular digraphs and near-regular Eulerian digraphs.

Section~2 develops the determinant identity that we use throughout.
Section~3 gives the short regular proof.  Section~4 proves the packing
theorems in the regular case, and Section~5 proves the Eulerian
existence and packing results.  Section~6 gives two sharpness
constructions, and Section~7 discusses open problems.

\section{The determinant method}

\subsection{Cycle covers and the permanent}

Fix a digraph $D$ with edge labels in a group $\Gamma$ of order $k$.
A \emph{directed cycle cover} of $D$ is a collection of pairwise
vertex-disjoint directed cycles of length at least two that covers every
vertex of $D$.
Number the vertices as $[n]:=\{1,\ldots,n\}$, and let $A=(a_{uv})$ be
the adjacency matrix, so $a_{uv}=1$ when $uv\in E(D)$ and $a_{uv}=0$
otherwise.  Its permanent is
\[
  \per A:=\sum_{\sigma\in S_n}\prod_{v=1}^n a_{v,\sigma(v)},
\]
where $S_n$ is the set of permutations of $[n]$.  A permutation
contributes $1$ exactly when its permutation cycles form a directed
cycle cover of $D$.  Thus $\per A$ counts the directed cycle covers of
$D$.  For $S\subseteq V(D)$, write $A[S]$ for
the principal submatrix on $S$.  Then $\per A[S]$ counts the directed
cycle covers of the subdigraph induced by $S$.  We set
$\per A[\varnothing]=1$.

We use the classical van der Waerden permanent bound in the regular
case.

\begin{lemma}[Van der Waerden]\label{lem:lower}
Let $M$ be a nonnegative $n\times n$ matrix whose row sums and column
sums all equal a positive number $d$.  Then
\[
  (\per M)^{1/n}>\frac de.
\]
\end{lemma}

\begin{proof}
$M/d$ is doubly stochastic, so the van der Waerden permanent
theorem~\cite{Egorychev,Falikman} gives
\[
  \per M\ge d^n\frac{n!}{n^n}>
  \left(\frac de\right)^n.
\]
\end{proof}

\subsection{Centered permutation matrices}

For $h\in\Gamma$, let $P_h$ be the $k\times k$ permutation matrix of
the map $x\mapsto hx$ on $\Gamma$.  Then
\[
  P_gP_h=P_{gh}.
\]
Let $J$ be the $k\times k$ all-ones matrix, and define
\[
  S_h:=P_h-\frac1kJ.
\]
These centered matrices retain the multiplicative rule and have the
trace values needed below.

\begin{lemma}\label{lem:centered}
For all $g,h\in\Gamma$,
\[
  S_gS_h=S_{gh},
  \qquad
  \tr S_h=
  \begin{cases}
    k-1,&h=1_\Gamma,\\
    -1,&h\ne1_\Gamma.
  \end{cases}
\]
Every row of every $S_h$ has squared Euclidean norm $(k-1)/k$.
\end{lemma}

\begin{proof}
Since $P_hJ=JP_h=J$ and $J^2=kJ$,
\[
  S_gS_h
  =\left(P_g-\frac1kJ\right)\left(P_h-\frac1kJ\right)
  =P_{gh}-\frac1kJ
  =S_{gh}.
\]
The permutation $P_1$ has trace $k$, whereas $P_h$ has trace zero for
$h\ne1_\Gamma$.  Since $\tr(J/k)=1$, this gives the trace formula.  A
row of $S_h$ has one entry $1-1/k$ and $k-1$ entries $-1/k$, so its
squared norm is $\left(1-\frac1k\right)^2+(k-1)\frac1{k^2}
  =\frac{k-1}{k}$.
\end{proof}

\subsection{The squarefree algebra}

We work in the commutative algebra
\[
  \mathcal A:=
  \mathbb C[z_1,\ldots,z_n]/(z_1^2,\ldots,z_n^2).
\]
For $S\subseteq[n]$, put
\[
  z_S:=\prod_{v\in S}z_v,
  \qquad z_\varnothing:=1.
\]
For a polynomial $f$, write $[z_S]f$ for the coefficient of $z_S$ in
$f$.
Every element of $\mathcal A$ has a unique expression
\[
  \sum_{S\subseteq[n]}a_Sz_S,
  \qquad a_S\in\mathbb C,
\]
and $z_Sz_T=z_{S\cup T}$ when $S\cap T=\varnothing$, while
$z_Sz_T=0$ otherwise.
Thus passing from an ordinary polynomial to its image in $\mathcal A$
discards the monomials divisible by some $z_v^2$ and does not change
any squarefree coefficient.

\subsection{The determinant--permanent correspondence}

Define the cycle-collection generating polynomial
\[
  \widetilde F(\mathbf z):=
  \prod_C(1+z_{V(C)}),
  \qquad \mathbf z:=(z_1,\ldots,z_n),
\]
where the product is over the simple directed cycles of $D$.  If two
cycles selected in the expansion intersect, the corresponding monomial
contains $z_v^2$ for some vertex $v$ and therefore vanishes in
$\mathcal A$.  Consequently,
expanding $\widetilde F$ selects collections of pairwise vertex-disjoint
cycles, and
\begin{equation}\label{eq:cycle-cover-permanent}
  [z_S]\widetilde F=\per A[S]
  \qquad(S\subseteq V(D)).
\end{equation}

When $D$ has no zero-sum directed cycle, our goal is to construct a
determinant polynomial $F$ such that
\[
  F=\widetilde F\qquad\text{in }\mathcal A.
\]
Since passing to $\mathcal A$ preserves squarefree coefficients, this
would express every $\per A[S]$ as a squarefree coefficient of a
determinant.  We can then bound these coefficients using Hadamard's
inequality.  We will use the following block-determinant identity.
\begin{lemma}\label{lem:cycle-product}
Assign a $q\times q$ complex matrix $B_{uv}$ to every edge $uv$ of
$D$, and form the $nq\times nq$ block matrix $B$ whose $uv$-block is
$B_{uv}$ when $uv\in E(D)$ and zero otherwise.  For a directed cycle
$C$, choose an initial vertex and write
$C=v_1v_2\cdots v_\ell v_1$ in its directed cyclic order.  With this
choice, put
\[
  B(C):=B_{v_1v_2}B_{v_2v_3}\cdots B_{v_\ell v_1}.
\]
Although the matrix $B(C)$ may depend on the initial vertex, cyclicity
of the trace implies that $\tr(B(C))$ does not.  Thus the quantity
appearing below is well defined.  Put
\[
  Z_q:=\diag(z_1I_q,\ldots,z_nI_q),
  \qquad
  F(\mathbf z):=\det(I_{nq}-Z_qB).
\]
Then, in $\mathcal A$,
\[
  F(\mathbf z)
  =
  \prod_C
  \left(1-\tr(B(C))z_{V(C)}\right),
\]
where the product is over the simple directed cycles of $D$.
Moreover, as an ordinary polynomial before passing to $\mathcal A$,
$\deg_{z_v}F\le q$ for every $v$.

In particular, take $q=k$ and assign $B_{uv}=S_{w(uv)}$ to every edge
$uv$.  If $D$ has no zero-sum directed cycle, then
\[
  [z_S]F=\per A[S]
  \qquad(S\subseteq V(D)).
\]
\end{lemma}

\begin{proof}
Put $M:=I_{nq}-Z_qB$, and index its rows and columns by pairs
$(v,a)\in[n]\times[q]$.  Thus $S_{nq}$ denotes the permutations of
these $nq$ pairs.  The determinant expansion is
\[
  F(\mathbf z)
  =
  \sum_{\pi\in S_{nq}}
  \operatorname{sgn}(\pi)
  \prod_{(v,a)\in[n]\times[q]}
  M_{(v,a),\pi(v,a)}.
\]
Since $D$ is loopless, $B_{vv}=0$.  Consequently,
\[
  M_{(v,a),\pi(v,a)}
  =
  \begin{cases}
    1,
      & \pi(v,a)=(v,a),\\[2mm]
    -z_v(B_{vu})_{ab},
      & \pi(v,a)=(u,b)\ne(v,a).
  \end{cases}
\]
If $\pi$ moves two pairs with the same first coordinate $v$, its
term is divisible by $z_v^2$ and therefore vanishes in $\mathcal A$.
Thus only permutations that move at most one pair $(v,a)$ for each
$v\in[n]$ survive in $\mathcal A$.  For any such permutation $\pi$ with a
nonzero term, let $U$ be the set of first coordinates of the pairs moved by
$\pi$.  There is a unique function $a\colon U\to[q]$ such that these pairs
are $(v,a(v))$, $v\in U$.  Since $\pi$ permutes the set of moved pairs, it
defines a permutation $\sigma$ of $U$ by
\[
  \pi(v,a(v))=(\sigma(v),a(\sigma(v)))
  \qquad(v\in U).
\]
Extend $\sigma$ to an element of $S_n$ by fixing every vertex outside $U$.
Call a permutation $\sigma\in S_n$ \emph{good} if
$v\sigma(v)\in E(D)$ at every nonfixed vertex.  The permutation obtained
above is good.
The corresponding nontrivial cycles of $\pi$ and $\sigma$ have the same
lengths, so
$\operatorname{sgn}(\pi)=\operatorname{sgn}(\sigma)$.

Conversely, let $\sigma$ be good, put
\[
  U_\sigma:=\{v\in V(D):\sigma(v)\ne v\},
\]
and choose a function $a\colon U_\sigma\to[q]$.  Define
\[
  \pi(v,a(v)):=(\sigma(v),a(\sigma(v)))
  \qquad(v\in U_\sigma)
\]
and let $\pi$ fix all remaining pairs.  This reverses the preceding
construction, and choices producing a zero matrix entry simply
contribute zero.  Hence, in $\mathcal A$,
\[
  F(\mathbf z)
  =
  \sum_{\sigma\ \mathrm{good}}
  \operatorname{sgn}(\sigma)
  \sum_{a\colon U_\sigma\to[q]}
  \prod_{v\in U_\sigma}
  M_{(v,a(v)),(\sigma(v),a(\sigma(v)))}.
\]

We continue to work in $\mathcal A$, where $z_v^2=0$ for every $v$.
Fix a good $\sigma$, and let $\mathcal K$ be its collection of
nontrivial cycles.  These cycles have disjoint vertex sets whose union is
$U_\sigma$, so choosing a function $a\colon U_\sigma\to[q]$ amounts to
choosing its restrictions to the cycles independently.  For each
$C\in\mathcal K$, write $C=v_1\cdots v_\ell v_1$ in directed cyclic order,
write $a_j$ for the value at $v_j$, and put $v_{\ell+1}:=v_1$ and
$a_{\ell+1}:=a_1$.  It follows that
\[
\begin{aligned}
  &\sum_{a\colon U_\sigma\to[q]}
  \prod_{v\in U_\sigma}
  M_{(v,a(v)),(\sigma(v),a(\sigma(v)))}\\
  &\quad=
  \prod_{C\in\mathcal K}
  \left(
    \sum_{(a_1,\ldots,a_\ell)\in[q]^\ell}
    \prod_{j=1}^{\ell}
    M_{(v_j,a_j),(v_{j+1},a_{j+1})}
  \right).
\end{aligned}
\]
Since $\operatorname{sgn}(\sigma)=\prod_{C\in\mathcal K}(-1)^{|C|-1}$, the
factor contributed by $C=v_1\cdots v_\ell v_1$ is

\[
\begin{aligned}
  &(-1)^{\ell-1}
  \sum_{(a_1,\ldots,a_\ell)\in[q]^\ell}
  \prod_{j=1}^{\ell}
  M_{(v_j,a_j),(v_{j+1},a_{j+1})}\\
  &\quad=(-1)^{\ell-1}(-1)^\ell z_{V(C)}
  \sum_{(a_1,\ldots,a_\ell)\in[q]^\ell}
  \prod_{j=1}^{\ell}
  (B_{v_jv_{j+1}})_{a_ja_{j+1}}\\
  &\quad=
  -z_{V(C)}
  \tr\left(B_{v_1v_2}B_{v_2v_3}\cdots B_{v_\ell v_1}\right)\\
  &\quad=
  -\tr(B(C))z_{V(C)}.
\end{aligned}
\]
Here $(-1)^{\ell-1}$ is the sign of the permutation cycle,
$(-1)^\ell$ comes from its $\ell$ entries of $-Z_qB$, and the remaining
sum is the trace of the displayed matrix product.

Since $\sigma$ is determined by $\mathcal K$ and fixes every remaining
vertex, summing over $\sigma$ and then expanding a product over all
directed cycles gives
\[
  F(\mathbf z)
  =
  \sum_{\mathcal K}
  \prod_{C\in\mathcal K}
  \left(-\tr(B(C))z_{V(C)}\right)
  =
  \prod_C\left(1-\tr(B(C))z_{V(C)}\right)
  \qquad\text{in }\mathcal A,
\]
where $\mathcal K$ ranges over collections of pairwise vertex-disjoint
directed cycles, including the empty collection.  Indeed, every term
in the last product that selects two intersecting cycles contains
$z_v^2$ for some $v$ and therefore vanishes in $\mathcal A$.

Finally, when $F$ is viewed as an ordinary polynomial, $z_v$ occurs
only in the $q$ rows of the block row indexed by $v$.  Multilinearity
of the determinant therefore gives $\deg_{z_v}F\le q$.

For the final statement, Lemma~\ref{lem:centered} gives
$B(C)=S_{w(C)}$ for every directed cycle $C$.  If $D$ has no zero-sum
cycle, then $\tr B(C)=-1$, so the identity just proved gives
\[
  F(\mathbf z)=\prod_C(1+z_{V(C)})
  =\widetilde F(\mathbf z)
  \qquad\text{in }\mathcal A.
\]
Equality in $\mathcal A$ implies equality of squarefree coefficients,
so~\eqref{eq:cycle-cover-permanent} gives
\[
  [z_S]F=[z_S]\widetilde F=\per A[S]
  \qquad(S\subseteq V(D)).\qedhere
\]
\end{proof}

\begin{remark}
The idea of using a determinant to encode collections of cycles goes
back to MacMahon's master theorem~\cite{MacMahon}.  Chu made explicit
the resulting connection between coefficients of determinants and
permanents~\cite{Chu}.  Kassel and L\'evy proved a block-matrix identity
in which a directed cycle contributes the trace of the product of the
matrices on its edges, namely the sum of its diagonal
entries~\cite[Theorem~3.1]{KasselLevy}.  Lemma~\ref{lem:cycle-product} is
the squarefree form of this identity used here.
\end{remark}

\begin{lemma}\label{lem:hadamard-coefficient}
Let $D$ be a $d$-regular digraph on $[n]$, where $d\ge1$, and let
$q\ge1$.  Assign a $q\times q$ complex matrix $B_{uv}$ to every edge
$uv$, and let $B$ be the $nq\times nq$ block matrix whose $uv$-block is
$B_{uv}$ when $uv\in E(D)$ and zero otherwise.  Put
\[
  Z_q:=\diag(z_1I_q,\ldots,z_nI_q),
  \qquad F(\mathbf z):=\det(I_{nq}-Z_qB).
\]
If every row of each $B_{uv}$ has Euclidean norm at most some $\rho>0$,
then
\[
  \left|[z_1\cdots z_n]F\right|^{1/n}
  <\rho\sqrt{eqd}.
\]
\end{lemma}

\begin{proof}
Lemma~\ref{lem:cycle-product} gives $\deg_{z_v}F\le q$ for every $v$.
First consider a polynomial
\[
  g(z)=c_0+c_1z+\cdots+c_qz^q.
\]
Let $u$ be uniform on the $(q+1)$st roots of unity.  For every integer
$a$,
\[
  \mathbb E[u^a]
  =
  \begin{cases}
    1,&q+1\mid a,\\
    0,&q+1\nmid a.
  \end{cases}
\]
Consequently, for every $s>0$,
\[
  s^{-1}\mathbb E\bigl[u^{-1}g(su)\bigr]=c_1,
\]
because after expansion the average of $u^{a-1}$ is zero for
$0\le a\le q$ unless $a=1$.  Applying this identity successively to
$z_1,\ldots,z_n$ gives
\[
  [z_1\cdots z_n]F
  =
  s^{-n}\mathbb E\left[
    \left(\prod_{v=1}^n u_v^{-1}\right)
    F(su_1,\ldots,su_n)
  \right],
\]
where $u_1,\ldots,u_n$ are independent and uniform on the
$(q+1)$st roots of unity.

Fix the $u_v$.  Every row of the matrix defining
$F(su_1,\ldots,su_n)$ is indexed by some $(v,a)$.  Besides its identity
entry, it contains the $a$th row of $-su_vB_{vw}$ for each edge $vw$
leaving $v$.  This row has Euclidean norm at most $s\rho$, since
$|u_v|=1$.  Simplicity and looplessness place the identity entry and
the $d$ matrix rows on disjoint coordinates.  Thus the squared norm of
the whole row is at most
\[
  1+\sum_{w:\,vw\in E(D)}s^2\rho^2=1+d\rho^2s^2.
\]
Hadamard's inequality and the preceding coefficient identity give
\[
  \left|[z_1\cdots z_n]F\right|
  \le
  s^{-n}(1+d\rho^2s^2)^{nq/2}.
\]
Taking $n$th roots and choosing $s=(qd\rho^2)^{-1/2}$, we obtain
\[
  \left|[z_1\cdots z_n]F\right|^{1/n}
  \le
  \rho\sqrt{qd}\left(1+\frac1q\right)^{q/2}
  <\rho\sqrt{eqd}.
\]
\end{proof}

\section{The regular case}

\begin{theorem}\label{thm:main}
Let $\Gamma$ be a finite group of order $k\ge2$.  Let $D$ be a
$d$-regular digraph with a $\Gamma$-labelling.  If
\[
  d\ge e^3(k-1),
\]
then $D$ contains a zero-sum directed cycle.
\end{theorem}

\begin{proof}
Suppose that $D$ has no zero-sum directed cycle.  Write
$V(D)=[n]:=\{1,\ldots,n\}$, let $A$ be the adjacency matrix of $D$,
and form the $nk\times nk$ block matrix $B$ by setting
$B_{uv}:=S_{w(uv)}$ for $uv\in E(D)$ and $B_{uv}:=0$ otherwise.
Lemma~\ref{lem:cycle-product} associates with these matrices the
determinant polynomial
\[
  F(\mathbf z):=\det(I_{nk}-Z_kB),
  \qquad
  Z_k:=\diag(z_1I_k,\ldots,z_nI_k).
\]
Since $A$ has all row and column sums equal to $d$,
Lemma~\ref{lem:lower} gives
\begin{equation}\label{eq:regular-lower}
  (\per A)^{1/n}>\frac de.
\end{equation}
On the other hand, Lemma~\ref{lem:cycle-product} gives
\[
  \per A=[z_1\cdots z_n]F.
\]
Lemma~\ref{lem:hadamard-coefficient} bounds this coefficient by
\[
  \left|[z_1\cdots z_n]F\right|^{1/n}
  <\rho\sqrt{eqd},
\]
where $q$ is the block size and $\rho$ bounds the row norms of the
matrices assigned to the edges.  Here those matrices are
$S_{w(uv)}$, so $q=k$ and $\rho=\sqrt{(k-1)/k}$.  Substitution gives
\[
  (\per A)^{1/n}
  <
  \sqrt{\frac{k-1}{k}}\sqrt{ekd}
  =
  \sqrt{e(k-1)d}.
\]
If $d\ge e^3(k-1)$, then the right-hand side is at most $d/e$,
contradicting~\eqref{eq:regular-lower}.
\end{proof}

\section{Packing zero-sum cycles}\label{sec:packing}

The cycles in a directed cycle cover are pairwise vertex-disjoint.  We
first show that some cover contains many zero-sum cycles.

\begin{theorem}[Vertex-disjoint zero-sum cycles]
\label{thm:vertex-packing}
Let $\Gamma$ be a finite group of order $k\ge2$.  Every $d$-regular
$\Gamma$-labelled digraph, with $d\ge1$, has a directed cycle cover
containing at least
\[
  \left\lfloor\frac{d}{25k}\right\rfloor
\]
zero-sum cycles.
\end{theorem}

\begin{proof}
Write $V(D)=[n]$, let $A$ be the adjacency matrix of $D$, and put
\[
  r:=\left\lfloor\frac{d}{25k}\right\rfloor.
\]
Lemma~\ref{lem:lower} gives $\per A>0$, so $D$ has a directed cycle
cover.
For a directed cycle cover $\mathcal C$, let $N_0(\mathcal C)$ be its
number of zero-sum cycles, and define
\[
  p(x):=\sum_{\mathcal C}x^{N_0(\mathcal C)},
\]
where the sum ranges over all directed cycle covers of $D$.  At
$x=1$, every cover has weight $1$, so
\[
  p(1)=\per A
  \qquad\text{and}\qquad
  p(1)^{1/n}>\frac de
\]
by Lemma~\ref{lem:lower}.

If $r=0$, any cycle cover proves the claim.  Suppose that $r\ge1$ and,
for a contradiction, that every directed cycle cover contains fewer
than $r$ zero-sum cycles.  Then $\deg p<r$.

To obtain a contradiction, we need an upper bound for $p(1)$.  Since
$\deg p<r$, we can express $p(1)$ in terms of the values $p(1-jk)$,
$1\le j\le r$, by interpolation.  We will then realize each
$p(1-jk)$ as the coefficient of $z_1\cdots z_n$ in a determinant
polynomial constructed from $jk\times jk$ matrices.  The coefficient
bound in Lemma~\ref{lem:hadamard-coefficient}, which follows from
Hadamard's inequality, will give the required estimates.

Put $f_0(a):=p(1-ak)$ and recursively define
$f_{i+1}(a):=f_i(a+1)-f_i(a)$.  This operation lowers the degree of a
nonconstant polynomial by one, so $f_r=0$ because $\deg f_0<r$.
Induction on $i$, using Pascal's identity, gives
$f_i(a)=\sum_{j=0}^i(-1)^{i-j}\binom ij f_0(a+j)$.  Taking $i=r$ and
$a=0$ therefore gives
\[
  0=
  \sum_{j=0}^r(-1)^{r-j}\binom rj p(1-jk).
\]
Isolating the term with $j=0$ yields
\begin{equation}\label{eq:cover-finite-difference}
  p(1)
  =
  \sum_{j=1}^r(-1)^{j-1}\binom rj p(1-jk).
\end{equation}

We already have a lower bound for $p(1)$, so it remains to obtain upper
bounds for the values $p(1-jk)$.  Fix $1\le j\le r$.  In $p(1-jk)$,
a zero-sum cycle has weight $1-jk$, while every other cycle has weight
$1$.  Lemma~\ref{lem:cycle-product} suggests assigning matrices
$B_{uv}$ to the edges so that every directed cycle
$C=v_1\cdots v_\ell v_1$ satisfies
\[
  -\tr(B_{v_1v_2}\cdots B_{v_\ell v_1})
  =
  \begin{cases}
    1-jk,&w(C)=1_\Gamma,\\
    1,&w(C)\ne1_\Gamma.
  \end{cases}
\]
Recall that
\[
  -\tr S_h=
  \begin{cases}
    1-k,&h=1_\Gamma,\\
    1,&h\ne1_\Gamma,
  \end{cases}
\]
whereas
\[
  \tr P_h=
  \begin{cases}
    k,&h=1_\Gamma,\\
    0,&h\ne1_\Gamma.
  \end{cases}
\]
To obtain the weight $1-jk$ instead of $1-k$, define the $jk\times jk$
block-diagonal matrix
\[
  R_h^{(j)}:=\diag(S_h,P_h,\ldots,P_h),
\]
with one copy of $S_h$ and $j-1$ copies of $P_h$.  Direct sums preserve
multiplication and add traces, so
\[
  R_g^{(j)}R_h^{(j)}=R_{gh}^{(j)},
  \qquad
  -\tr R_h^{(j)}
  =
  \begin{cases}
    1-jk,&h=1_\Gamma,\\
    1,&h\ne1_\Gamma.
  \end{cases}
\]
Every row of $R_h^{(j)}$ has Euclidean norm at most $1$, because rows
of $P_h$ have norm $1$ and rows of $S_h$ have norm
$\sqrt{(k-1)/k}$.

Let $B^{(j)}$ be the block matrix whose $uv$-block is
$R_{w(uv)}^{(j)}$ when $uv\in E(D)$ and zero otherwise.  Put
\[
  F_j(\mathbf z):=\det(I_{njk}-Z_{jk}B^{(j)}).
\]
For every directed cycle $C$, multiplicativity gives
$B^{(j)}(C)=R_{w(C)}^{(j)}$.  Lemma~\ref{lem:cycle-product}
therefore gives, in the squarefree algebra,
\[
  F_j(\mathbf z)
  =\prod_C\left(1-\tr R_{w(C)}^{(j)}z_{V(C)}\right).
\]
Equality in the squarefree algebra means that the two sides have the
same squarefree coefficients.  The coefficient of
$z_1\cdots z_n$ retains exactly the collections of vertex-disjoint
cycles covering every vertex, namely the directed cycle covers.
Thus, when $F_j$ is viewed as an ordinary polynomial,
\[
\begin{aligned}
  [z_1\cdots z_n]F_j
  &=
  \sum_{\mathcal C}
  \prod_{C\in\mathcal C}
  \left(-\tr R_{w(C)}^{(j)}\right)\\
  &=
  \sum_{\mathcal C}(1-jk)^{N_0(\mathcal C)}
  =p(1-jk).
\end{aligned}
\]
The first sum is over all directed cycle covers $\mathcal C$.  Its
product is $(1-jk)^{N_0(\mathcal C)}$ because each zero-sum cycle
contributes $1-jk$ and every other cycle contributes $1$.

Lemma~\ref{lem:hadamard-coefficient}, with block size $jk$ and
row-norm bound $1$, now gives
\begin{equation}\label{eq:cover-evaluation-bound}
  |p(1-jk)|^{1/n}
  =\left|[z_1\cdots z_n]F_j\right|^{1/n}
  <\sqrt{ejkd}.
\end{equation}

Taking absolute values in~\eqref{eq:cover-finite-difference},
using~\eqref{eq:cover-evaluation-bound}, and recalling that $j\le r$, we
obtain
\[
\begin{aligned}
  p(1)
  &\le
  \sum_{j=1}^r\binom rj|p(1-jk)|\\
  &<
  \sum_{j=1}^r\binom rj
  \left(\sqrt{erkd}\right)^n\\
  &<
  2^r\left(\sqrt{erkd}\right)^n.
\end{aligned}
\]
Taking $n$th roots and comparing this with the lower bound for $p(1)$
gives
\[
  \frac de
  <p(1)^{1/n}
  <2^{r/n}\sqrt{erkd}.
\]
Recall that $r=\lfloor d/(25k)\rfloor$.  Since $D$ is simple and
loopless, $d<n$.  Thus $r/n<1/(25k)\le1/50$ and $rkd\le d^2/25$, so
\[
  p(1)^{1/n}
  <2^{1/50}\sqrt{\frac e{25}}\,d
  <\frac de.
\]
This contradiction proves that some cycle cover contains at least
$r=\lfloor d/(25k)\rfloor$ zero-sum cycles.
The zero-sum cycles in the cover are pairwise vertex-disjoint.
\end{proof}

\begin{theorem}[Edge-disjoint zero-sum cycles]\label{thm:packing}
Let $\Gamma$ be a finite group of order $k\ge2$.  Let $D$ be a
$d$-regular digraph with a $\Gamma$-labelling.  If $d\ge50k$, then $D$
contains at least
\[
  \frac{d^2}{100k}
\]
pairwise edge-disjoint zero-sum directed cycles.
\end{theorem}

\begin{proof}
Set $D_d:=D$.  For $m=d,d-1,\ldots,1$, apply
Theorem~\ref{thm:vertex-packing} to the $m$-regular digraph $D_m$.
Choose a directed cycle cover containing at least
$\left\lfloor\frac{m}{25k}\right\rfloor$ zero-sum cycles, retain those
cycles, and delete every edge of the cover.  Since the cover uses
exactly one incoming and one outgoing
edge at each vertex, the remaining digraph $D_{m-1}$ is
$(m-1)$-regular.  Every cycle retained in a later round uses only
edges that remain after the earlier covers have been deleted.
Therefore all the retained zero-sum cycles are pairwise edge-disjoint,
and their number is at least
\[
  \sum_{m=1}^d\left\lfloor\frac{m}{25k}\right\rfloor
  \ge\frac{d^2}{100k}.
\]
\end{proof}

\section{Eulerian digraphs}\label{sec:eulerian}

In the regular proof, $A/d$ is doubly stochastic, so the van der Waerden
theorem gives a lower bound for $\per A$.  For an irregular Eulerian
digraph, however, $\per A$ may be zero.  We therefore add diagonal
weights and instead lower-bound the permanent of
$M_{\mathbf t}=A+\diag(t_1,\ldots,t_n)$.  This matrix need not become
doubly stochastic after scalar normalization, so we use the stronger
Bethe permanent lower bound.

\begin{samepage}
\begin{theorem}\label{thm:eulerian}
Let $\Gamma$ be a finite group of order $k\ge2$.  Let $D$ be an
Eulerian digraph with a $\Gamma$-labelling, and write
\[
  d(v):=d^+(v)=d^-(v),\qquad
  \delta:=\min_v d(v),\qquad
  \Delta:=\max_v d(v).
\]
If $\Delta>0$ and
\[
  \frac{\delta^3}{\Delta^2}\ge e^3(k-1),
\]
then $D$ contains a zero-sum directed cycle.
\end{theorem}
\end{samepage}

\subsection{Weighted cycle collections}

Let $V(D)=[n]$, let $A$ be the adjacency matrix of $D$, and choose
nonnegative weights $\mathbf t=(t_1,\ldots,t_n)$.  Put
\[
M_{\mathbf t}:=A+\diag(t_1,\ldots,t_n).
\]
Expanding the permanent according to the diagonal entries $t_v$ that
are chosen gives
\begin{equation}\label{eq:weighted-permanent-expansion}
  \per M_{\mathbf t}
  =
  \sum_{S\subseteq[n]}
  \left(\prod_{v\notin S}t_v\right)\per A[S].
\end{equation}
Suppose that $D$ has no zero-sum directed cycle, and let $F$ be the
determinant polynomial obtained in Lemma~\ref{lem:cycle-product} by
assigning $S_{w(uv)}$ to every edge $uv$.  That lemma gives
$[z_S]F=\per A[S]$ for every $S$.  Hence
\begin{equation}\label{eq:weighted-permanent-determinant}
  \per M_{\mathbf t}
  =
  \sum_{S\subseteq[n]}
  \left(\prod_{v\notin S}t_v\right)[z_S]F.
\end{equation}
We will lower-bound $\per M_{\mathbf t}$ directly and obtain an upper
bound from determinant evaluations.

\subsection{A permanent upper bound}

\begin{lemma}
\label{lem:weighted-bound}
Let $\Gamma$ be a finite group of order $k\ge2$, and let $D$ be a
digraph with vertex set $[n]$ and a $\Gamma$-labelling.  Suppose that
$D$ has no zero-sum directed cycle.  Let $A$ be the adjacency matrix
of $D$, choose nonnegative weights
$\mathbf t=(t_1,\ldots,t_n)$, and put
\[
  M_{\mathbf t}:=A+\diag(t_1,\ldots,t_n).
\]
Then, for every $s>0$,
\[
  \per M_{\mathbf t}
  \le
  \prod_{v=1}^n
  \sqrt{t_v^2+s^{-2}}
  \left(1+\frac{k-1}{k}d_D^+(v)s^2\right)^{k/2}.
\]
\end{lemma}

\begin{proof}
Let $B$ be the block matrix obtained by assigning $S_{w(uv)}$ to every
edge $uv$, and let $F$ be the corresponding determinant polynomial
from Lemma~\ref{lem:cycle-product}.  By~\eqref{eq:weighted-permanent-determinant},
\[
  \per M_{\mathbf t}
  =
  \sum_{S\subseteq[n]}
  \left(\prod_{v\notin S}t_v\right)[z_S]F.
\]
Each variable $z_v$ occurs only in the $k$ rows of the corresponding
block row of the determinant, so $F$ has degree at most $k$ in each
variable.

Let $u_1,\ldots,u_n$ be independent and uniform on the $(k+1)$st
roots of unity.  As in the proof of
Lemma~\ref{lem:hadamard-coefficient}, for $0\le a\le k$ we have
\[
  \mathbb E\left[(t_v+s^{-1}u_v^{-1})(su_v)^a\right]
  =
  \begin{cases}
    t_v,&a=0,\\
    1,&a=1,\\
    0,&2\le a\le k.
\end{cases}
\]
Indeed, after expansion the two powers of $u_v$ are $a$ and $a-1$,
and averaging over the roots of unity kills every exponent in the
range $-1,\ldots,k$ except $0$.  Applying this identity independently
to the preceding coefficient sum gives
\begin{equation}\label{eq:weighted-extraction}
  \per M_{\mathbf t}
  =\mathbb E\left[
    \prod_{v=1}^n(t_v+s^{-1}u_v^{-1})
    F(su_1,\ldots,su_n)
  \right].
\end{equation}

Fix a choice of $u_1,\ldots,u_n$.  At these values, the polynomial is
being evaluated as the determinant
\[
  F(su_1,\ldots,su_n)
  =
  \det\left(
    I_{nk}-\diag(su_1I_k,\ldots,su_nI_k)B
  \right).
\]
Consider a scalar row belonging to the block row indexed by $v$.  The
diagonal block $I_k$ contributes one entry of squared modulus $1$.
For every edge $vw$, the same scalar row contains one row of
$-su_vS_{w(vw)}$, which has squared norm $s^2(k-1)/k$.
Because $D$ is loopless, the identity block and the edge blocks lie in
different block columns.  Because $D$ is simple, the outgoing edges
from $v$ also give distinct block columns.  Hence the scalar row has
squared norm at most
\[
  1+\frac{k-1}{k}d_D^+(v)s^2.
\]
There are $k$ such rows for each vertex.  Hadamard's inequality
therefore gives the following bound for every choice of
$u_1,\ldots,u_n$:
\[
  |F(su_1,\ldots,su_n)|
  \le
  \prod_{v=1}^n
  \left(1+\frac{k-1}{k}d_D^+(v)s^2\right)^{k/2}.
\]

Substituting this pointwise bound into~\eqref{eq:weighted-extraction},
taking absolute values, and using
independence, we obtain
\[
  \per M_{\mathbf t}
  \le
  \prod_{v=1}^n
  \left(1+\frac{k-1}{k}d_D^+(v)s^2\right)^{k/2}
  \prod_{v=1}^n\mathbb E|t_v+s^{-1}u_v^{-1}|.
\]
Since $\mathbb E u_v=\mathbb E u_v^{-1}=0$, the Cauchy--Schwarz
inequality gives
\[
  \mathbb E|t_v+s^{-1}u_v^{-1}|
  \le
  \sqrt{\mathbb E|t_v+s^{-1}u_v^{-1}|^2}
  =\sqrt{t_v^2+s^{-2}}.
\]
Multiplying these bounds proves the lemma.
\end{proof}

\begin{remark}
The estimate in Lemma~\ref{lem:weighted-bound} follows the same pattern
as Lossers' solution~\cite{Lossers} of a problem of Goldstein and
Graham~\cite{GoldsteinGraham}: average determinant evaluations over
roots of unity to select the desired coefficients, and then apply
Hadamard's inequality.
\end{remark}

\subsection{Existence}

For the lower bound we use the following form of an inequality proved
by Gurvits using Schrijver's permanental inequality and the Bethe
approximation~\cite{GurvitsBethe}.

\begin{lemma}[Bethe permanent inequality]\label{lem:bethe}
Let $P$ be a nonnegative $n\times n$ matrix, and let $Q$ be a doubly
stochastic matrix such that $Q_{ij}>0$ only if $P_{ij}>0$.  Then
\[
  \per P
  \ge
  \prod_{\substack{i,j\\Q_{ij}>0}}
  \left(\frac{P_{ij}}{Q_{ij}}\right)^{Q_{ij}}
  (1-Q_{ij})^{1-Q_{ij}},
\]
where $0^0$ in the second factor is interpreted as $1$.
\end{lemma}

\begin{proof}[Proof of Theorem~\ref{thm:eulerian}]
Suppose that $D$ has no zero-sum directed cycle.  Choose nonnegative
weights $t_1,\ldots,t_n$, to be specified later, with $t_v>0$ whenever
$d(v)<\Delta$.  We compare a lower bound for $\per M_{\mathbf t}$ with
the determinant upper bound supplied by Lemma~\ref{lem:weighted-bound}.

\medskip\noindent
\emph{The permanent lower bound.}
To apply Lemma~\ref{lem:bethe} with $P=M_{\mathbf t}$, we need a
doubly stochastic matrix $Q$ whose positive entries occur only where
$M_{\mathbf t}$ is positive.  Put
\[
  p_v:=\frac{d(v)}\Delta.
\]
Since $\delta^3/\Delta^2\le\delta$ and $e^3(k-1)>1$, the hypothesis
implies $\delta>1$.  In particular, $0<p_v\le1$ and $\Delta\ge2$.
Define the matrix $Q$ by
\[
  Q_{uv}:=
  \begin{cases}
    1/\Delta,&uv\in E(D),\\
    1-p_v,&u=v,\\
    0,&\text{otherwise}.
  \end{cases}
\]
Since $D$ is Eulerian, the row and column sums of $Q$ at $v$ are both
$p_v+1-p_v=1$.
Thus $Q$ is doubly stochastic.
Its positive off-diagonal entries correspond to edges of $D$, where
$M_{\mathbf t}$ has entry $1$.  If $Q_{vv}>0$, then $p_v<1$, so
$d(v)<\Delta$ and therefore $t_v>0$.  Thus every positive entry of
$Q$ corresponds to a positive entry of $M_{\mathbf t}$, as required by
Lemma~\ref{lem:bethe}.

We apply Lemma~\ref{lem:bethe} with $P=M_{\mathbf t}$ and with the
matrix $Q$ defined above.  For each of the $d(v)=p_v\Delta$ edges
leaving $v$, the corresponding entries of $P$ and $Q$ are $1$ and
$1/\Delta$.  Since
\[
  \left(1-\frac1\Delta\right)^{\Delta-1}>e^{-1},
\]
their combined contribution to the Bethe product is
\[
  \Delta^{p_v}
  \left(1-\frac1\Delta\right)^{p_v(\Delta-1)}
  >\left(\frac{\Delta}{e}\right)^{p_v}.
\]
If $p_v<1$, then $Q_{vv}=1-p_v>0$, and the diagonal entries
$P_{vv}=t_v$ and $Q_{vv}=1-p_v$ contribute
\[
  \left(\frac{t_v}{1-p_v}\right)^{1-p_v}p_v^{p_v}.
\]
If $p_v=1$, then $Q_{vv}=0$.  The product in
Lemma~\ref{lem:bethe} ranges only over entries for which $Q_{ij}>0$,
so there is no diagonal factor at $v$.  Multiplying the edge and
diagonal contributions gives
\begin{equation}\label{eq:bethe-lower}
  \per M_{\mathbf t}
  >
  \prod_v
  \left(\frac{\Delta}{e}\right)^{p_v}p_v^{p_v}
  \prod_{v:p_v<1}
  \left(\frac{t_v}{1-p_v}\right)^{1-p_v}.
\end{equation}

\medskip\noindent
\emph{The determinant upper bound.}
Lemma~\ref{lem:weighted-bound} gives, for every $s>0$,
\begin{equation}\label{eq:eulerian-upper}
  \per M_{\mathbf t}
  \le
  \prod_v
  \sqrt{t_v^2+s^{-2}}
  \left(1+\frac{k-1}{k}d(v)s^2\right)^{k/2}.
\end{equation}

\medskip\noindent
\emph{Choosing the parameters.}
Set
\[
  s^{-2}:=\Delta(k-1),
  \qquad
  t_v
  :=s^{-1}\sqrt{\frac{1-p_v}{p_v}}.
\]
For each $v$, write $p=p_v$, and let $R_v$ be the ratio of the
corresponding factors in the right-hand
sides of~\eqref{eq:bethe-lower} and~\eqref{eq:eulerian-upper}, with the
diagonal factor in~\eqref{eq:bethe-lower} omitted when $p=1$.
For $p<1$, the definition of $R_v$ gives
\begin{align*}
  R_v
  &=
  \frac{(\Delta/e)^p p^p
    \left(t_v/(1-p)\right)^{1-p}}
  {\sqrt{t_v^2+s^{-2}}
    \left(1+\frac{k-1}{k}d(v)s^2\right)^{k/2}}\\
  &=e^{-p}
  \left(\frac{d(v)^3}{\Delta^2(k-1)}\right)^{p/2}
  (1-p)^{-(1-p)/2}
  (1+p/k)^{-k/2}.
\end{align*}
The second equality follows by substituting $d(v)=p\Delta$ and the
chosen values of $s$ and $t_v$.
When $p=1$, the diagonal factor is absent and $t_v=0$.  Direct
substitution gives the same expression, with the factor involving
$1-p$ interpreted as $1$.  By the hypothesis,
\[
  \left(\frac{d(v)^3}{\Delta^2(k-1)}\right)^{p/2}
  \ge e^{3p/2}.
\]
Moreover, $(1-p)^{-(1-p)/2}\ge1$, while $1+x<e^x$ gives
$(1+p/k)^{-k/2}>e^{-p/2}$.  Hence
\[
  R_v>e^{-p}e^{3p/2}e^{-p/2}=1.
\]
Thus the right-hand side of~\eqref{eq:bethe-lower} is larger than that
of~\eqref{eq:eulerian-upper}, a contradiction.
\end{proof}

\begin{corollary}\label{cor:all-moduli}
Let $k\ge2$.  If $d\ge e^3(k-1)$, then every $d$-regular digraph
contains pairwise edge-disjoint directed cycles
\[
  C_2,C_3,\ldots,C_k
\]
such that the length of $C_r$ is divisible by $r$ for every
$2\le r\le k$.
\end{corollary}

\begin{proof}
Starting with $r=k$ and proceeding downward, we find $C_r$ and delete
its edges.  Suppose that $t=k-r$ cycles have already been deleted, and
let $D_r$ be the remaining digraph.  Deleting the edges of a directed
cycle decreases the indegree and outdegree of each of its vertices by
one, so $D_r$ is Eulerian.  If $\delta_r$ and $\Delta_r$ are its minimum
and maximum common degrees, then
\[
  \delta_r\ge d-t,
  \qquad
  \Delta_r\le d.
\]
Since $t\le k-2<d$, we have
\[
  \frac{\delta_r^3}{\Delta_r^2}
  \ge \frac{(d-t)^3}{d^2}
  \ge d-3t
  \ge e^3(k-1)-3t
  \ge e^3(r-1).
\]
Label every edge of $D_r$ by $1\in\mathbb Z_r$.  Theorem~\ref{thm:eulerian}
gives a zero-sum directed cycle $C_r$, whose length is divisible by
$r$.  Deleting its edges leaves an Eulerian digraph for the next step.
The resulting cycles are pairwise edge-disjoint.
\end{proof}

\subsection{Packing zero-sum cycles}

\begin{theorem}[Vertex-disjoint zero-sum cycles]
\label{thm:eulerian-vertex-packing}
Let $\Gamma$ be a finite group of order $k\ge2$.  Let $D$ be an
Eulerian digraph with a $\Gamma$-labelling and minimum and maximum
common degrees $\delta$ and $\Delta>0$.  Then $D$ contains at least
\[
  \left\lfloor\frac{\delta^3}{4e^3k\Delta^2}\right\rfloor
\]
pairwise vertex-disjoint zero-sum directed cycles.
\end{theorem}

\begin{proof}
Put
\[
  r:=\left\lfloor\frac{\delta^3}{4e^3k\Delta^2}\right\rfloor.
\]
There is nothing to prove if $r=0$, so suppose that $r\ge1$.

\medskip\noindent
\emph{The generating polynomial.}
Write $V(D)=[n]$, let $A$ be the adjacency matrix of $D$, and put
\[
  p_v:=\frac{d(v)}\Delta,
  \qquad \lambda:=rk,
  \qquad s^{-2}:=\Delta\lambda,
  \qquad
  t_v:=s^{-1}\sqrt{\frac{1-p_v}{p_v}}.
\]
The assumption $r\ge1$ implies $\delta>1$, so these quantities are
well defined.  The weight $t_v$ is zero precisely when $p_v=1$.
Set
\[
  M_{\mathbf t}:=A+\diag(t_1,\ldots,t_n).
\]

Let $\mathcal C$ range over collections of pairwise vertex-disjoint
directed cycles, including the empty collection.  Write $V(\mathcal C)$
for the union of their vertex sets and $N_0(\mathcal C)$ for the number
of zero-sum cycles in $\mathcal C$, and define
\[
  P(x):=
  \sum_{\mathcal C}x^{N_0(\mathcal C)}
  \prod_{v\notin V(\mathcal C)}t_v.
\]
At $x=1$, group the cycle collections according to their covered
vertex set $S$.  Equations~\eqref{eq:cycle-cover-permanent}
and~\eqref{eq:weighted-permanent-expansion} give
\begin{align*}
  P(1)
  &=\sum_{\mathcal C}\prod_{v\notin V(\mathcal C)}t_v\\
  &=\sum_{S\subseteq[n]}
    \left(\prod_{v\notin S}t_v\right)\per A[S]\\
  &=\per M_{\mathbf t}.
\end{align*}
Suppose for a contradiction that $D$ contains fewer than $r$
pairwise vertex-disjoint zero-sum cycles.  Then $\deg P<r$, so the
finite-difference identity~\eqref{eq:cover-finite-difference} gives
\begin{equation}\label{eq:eulerian-finite-difference}
  P(1)=\sum_{j=1}^r(-1)^{j-1}\binom rj P(1-jk).
\end{equation}
We will prove that, for every $1\le j\le r$,
\begin{equation}\label{eq:eulerian-shifted-target}
  |P(1-jk)|<2^{-r}P(1).
\end{equation}
Since the sum of the binomial coefficients
in~\eqref{eq:eulerian-finite-difference} is $2^r-1$, this estimate will
give the desired contradiction.

\medskip\noindent
\emph{Bounding $P(1-jk)$.}
Fix $1\le j\le r$.  Just as in the proof of
Theorem~\ref{thm:vertex-packing}, for $h\in\Gamma$, let $P_h$ be the
$k\times k$ permutation matrix of the map $x\mapsto hx$, let $J$ be
the $k\times k$ all-ones matrix, and put $S_h:=P_h-J/k$.  Let
$R_h^{(j)}$ be the block-diagonal matrix with one block $S_h$ and
$j-1$ blocks $P_h$.
Then
\[
  R_g^{(j)}R_h^{(j)}=R_{gh}^{(j)},
  \qquad
  -\tr R_h^{(j)}=
  \begin{cases}
    1-jk,&h=1_\Gamma,\\
    1,&h\ne1_\Gamma.
  \end{cases}
\]
Let $B^{(j)}$ be the block matrix whose $uv$-block is
$R_{w(uv)}^{(j)}$ when $uv\in E(D)$ and zero otherwise, and define
\[
  F_j(\mathbf z)
  :=\det\left(
    I_{njk}-\diag(z_1I_{jk},\ldots,z_nI_{jk})B^{(j)}
  \right).
\]
By Lemma~\ref{lem:cycle-product}, in $\mathcal A$ we have
$F_j(\mathbf z)=\prod_C(1-\tr(B^{(j)}(C))z_{V(C)})$, where the product
is over the simple directed cycles of $D$.  For each such cycle $C$,
the multiplicative rule above gives
$B^{(j)}(C)=R_{w(C)}^{(j)}$.  Thus the coefficient of $z_{V(C)}$ in
the factor indexed by $C$ is $1-jk$ if $C$ is zero-sum and $1$
otherwise.  Since $z_v^2=0$ in $\mathcal A$, only products indexed by
vertex-disjoint cycles survive.  The surviving products that contribute
to $[z_S]F_j$ are therefore precisely the directed cycle covers
$\mathcal C$ of the subdigraph induced by $S$.  Hence
\[
  [z_S]F_j
  =\sum_{\mathcal C\text{ a directed cycle cover of }S}
  (1-jk)^{N_0(\mathcal C)}.
\]
Consequently,
\[
  P(1-jk)
  =
  \sum_{S\subseteq[n]}
  \left(\prod_{v\notin S}t_v\right)[z_S]F_j.
\]
Each variable has degree at most $jk$ in $F_j$.  Let
$u_1,\ldots,u_n$ be independent and uniform on the $(jk+1)$st roots
of unity.  For $0\le a\le jk$,
\[
  \mathbb E\left[(t_v+s^{-1}u_v^{-1})(su_v)^a\right]
  =
  \begin{cases}
    t_v,&a=0,\\
    1,&a=1,\\
    0,&2\le a\le jk.
  \end{cases}
\]
Applying this identity independently in every variable gives
\[
  P(1-jk)
  =\mathbb E\left[
    \prod_{v=1}^n(t_v+s^{-1}u_v^{-1})
    F_j(su_1,\ldots,su_n)
  \right].
\]
Every row of $R_h^{(j)}$ has norm at most $1$.  For fixed
$u_1,\ldots,u_n$, the identity entry and the $d(v)$ edge-block rows
in a scalar row indexed by $v$ lie in distinct block columns because
$D$ is simple and loopless.  That row therefore has squared norm at
most $1+d(v)s^2$, and Hadamard's inequality gives
\[
  |F_j(su_1,\ldots,su_n)|
  \le\prod_v(1+d(v)s^2)^{jk/2}.
\]
Moreover,
$\mathbb E|t_v+s^{-1}u_v^{-1}|\le\sqrt{t_v^2+s^{-2}}$ by
Cauchy--Schwarz.  Thus, setting
$U_{v,j}:=\sqrt{t_v^2+s^{-2}}(1+d(v)s^2)^{jk/2}$ and using
independence, we obtain
\begin{equation}\label{eq:eulerian-packing-upper}
  |P(1-jk)|\le\prod_vU_{v,j}.
\end{equation}

\medskip\noindent
\emph{Comparing the bounds.}
For each vertex $v$, define
\[
  L_v:=
  \left(\frac{\Delta}{e}\right)^{p_v}p_v^{p_v}
  \begin{cases}
    \left(\dfrac{t_v}{1-p_v}\right)^{1-p_v},&p_v<1,\\[1.2ex]
    1,&p_v=1.
  \end{cases}
\]
The derivation of the Bethe lower bound~\eqref{eq:bethe-lower} applies
to the present weights because $t_v>0$ whenever $p_v<1$.  Since
$P(1)=\per M_{\mathbf t}$, it gives
\[
  P(1)>\prod_vL_v.
\]

To prove $|P(1-jk)|<2^{-r}P(1)$, the lower bound above
and~\eqref{eq:eulerian-packing-upper} show that it suffices to prove
\[
  \prod_vU_{v,j}<2^{-r}\prod_vL_v.
\]
Since $D$ is simple and loopless, $\Delta<n$.  Moreover,
$\delta^3/\Delta^2\le\delta$, so the definition of $r$ gives
$r<\delta$.  Hence
\[
  \sum_vp_v
  =\frac{1}{\Delta}\sum_vd(v)
  \ge\frac{n\delta}{\Delta}
  >\delta
  >r.
\]
It is therefore enough to prove
$L_v/U_{v,j}>2^{p_v}$ for every vertex $v$.

Fix a vertex $v$ and write $p=p_v$.  Substituting the choices of $s$
and $t_v$ into the displayed formulas for $L_v$ and $U_{v,j}$ gives
\[
  \frac{L_v}{U_{v,j}}
  =e^{-p}
  \left(\frac{d(v)^3}{\Delta^2\lambda}\right)^{p/2}
  (1-p)^{-(1-p)/2}
  \left(1+\frac p\lambda\right)^{-jk/2},
\]
where the factor involving $1-p$ is interpreted as $1$ when $p=1$.
Since $jk\le rk=\lambda$ and $1+x<e^x$, we have
\[
  \left(1+\frac p\lambda\right)^{-jk/2}>e^{-p/2}.
\]
Also $(1-p)^{-(1-p)/2}\ge1$.  Since $d(v)\ge\delta$ and
$r\le\delta^3/(4e^3k\Delta^2)$, we also have
$d(v)^3/(e^3\Delta^2rk)\ge4$.  Consequently,
\[
  \frac{L_v}{U_{v,j}}
  >\left(\frac{d(v)^3}{e^3\Delta^2rk}\right)^{p_v/2}
  \ge2^{p_v}.
\]
Multiplying these inequalities and using $\sum_vp_v>r$, the two product
bounds give
\[
  |P(1-jk)|
  \le\prod_vU_{v,j}
  <2^{-\sum_vp_v}\prod_vL_v
  <2^{-r}\prod_vL_v
  <2^{-r}P(1).
\]
Thus~\eqref{eq:eulerian-shifted-target} holds.  Since $j$ was
arbitrary, it holds for every $1\le j\le r$.  Taking absolute values
in~\eqref{eq:eulerian-finite-difference} now gives
\[
  P(1)
  \le\sum_{j=1}^r\binom rj|P(1-jk)|
  <(2^r-1)2^{-r}P(1),
\]
a contradiction.
\end{proof}

\begin{theorem}[Edge-disjoint zero-sum cycles]
\label{thm:eulerian-edge-packing}
Let $\Gamma$ be a finite group of order $k\ge2$.  Let $D$ be an
Eulerian digraph with a $\Gamma$-labelling and minimum and maximum
common degrees $\delta$ and $\Delta>0$.  If
\[
  \frac{\delta^3}{\Delta^2}\ge32e^3k,
\]
then $D$ contains at least
\[
  \frac{\delta^4}{128e^3k\Delta^2}
\]
pairwise edge-disjoint zero-sum directed cycles.
\end{theorem}

\begin{proof}
Starting with $D_0:=D$, repeatedly apply
Theorem~\ref{thm:eulerian-vertex-packing} and delete the edges of the
zero-sum cycles it supplies.  Since the cycles chosen in each round are
vertex-disjoint, every common degree decreases by at most one.  The
remaining digraph is Eulerian, and after $t$ rounds its minimum and
maximum common degrees $\delta_t$ and $\Delta_t$ satisfy
\[
  \delta_t\ge\delta-t,
  \qquad
  \Delta_t\le\Delta.
\]
For every $0\le t<\lceil\delta/2\rceil$, we have
$\delta_t>\delta/2$.  The number of cycles found in round $t$ is
therefore at least
\[
  \left\lfloor
  \frac{\delta_t^3}{4e^3k\Delta_t^2}
  \right\rfloor
  \ge
  \left\lfloor
  \frac{\delta^3}{32e^3k\Delta^2}
  \right\rfloor
  \ge
  \frac{\delta^3}{64e^3k\Delta^2}.
\]
The last inequality uses $\lfloor x\rfloor\ge x/2$ for $x\ge1$,
which applies by the hypothesis.  Cycles obtained in different rounds
are edge-disjoint.  Since $\lceil\delta/2\rceil\ge\delta/2$, their
total number is at least
\[
  \frac\delta2\cdot
  \frac{\delta^3}{64e^3k\Delta^2}
  =\frac{\delta^4}{128e^3k\Delta^2}.
\]
\end{proof}

\section{Sharpness constructions}

Proposition~\ref{prop:obstruction} gives the sharp order for both
regular packing bounds.  Its underlying digraph is the
balanced blow-up of a directed cycle, a standard extremal construction
for controlling directed cycle lengths~\cite[p.~252]{KellyKuhnOsthus}.
Since regular digraphs are Eulerian, the proposition also shows that
the Eulerian packing bounds have the correct order when
$\delta=\Delta$.  Proposition~\ref{prop:circulant-obstruction} gives a
$k$-regular digraph with no zero-sum cycle for $\mathbb Z_k$.

\begin{proposition}\label{prop:obstruction}
Let $k\ge2$, $d\ge1$, and let $t\ge2$ be relatively prime to $k$.
There is a $d$-regular digraph with constant edge label
$1\in\mathbb Z_k$ in which every zero-sum cycle has length at least
$tk$.  Consequently, it contains at most $\lfloor d/k\rfloor$ pairwise
vertex-disjoint zero-sum cycles and at most $\lfloor d^2/k\rfloor$
pairwise edge-disjoint zero-sum cycles.  If $d<k$, it has no zero-sum
cycle.
\end{proposition}

\begin{proof}
Take disjoint vertex classes
\[
  V_0,V_1,\ldots,V_{t-1},
  \qquad |V_i|=d,
\]
and include every edge from $V_i$ to $V_{i+1}$, taking subscripts
modulo $t$.  Each vertex has exactly $d$ outgoing edges and $d$ incoming
edges, so the digraph is $d$-regular.  It is simple and loopless.

Every directed cycle has length divisible by $t$, because each edge
advances by one vertex class.  Since every edge has label
$1\in\mathbb Z_k$, a zero-sum cycle also has length divisible by $k$.
The coprimality of $t$ and $k$ therefore makes its length divisible by
$tk$.

The digraph has $td$ vertices, while every zero-sum cycle uses at least
$tk$ vertices.  Hence a vertex-disjoint collection has size at most
$td/(tk)=d/k$.  It also has $td^2$ edges, so an edge-disjoint
collection has size at most $td^2/(tk)=d^2/k$.  Finally, a simple cycle
has length at most the total number $td$ of vertices.  If $d<k$, then
$td<tk$, so no zero-sum cycle exists.
\end{proof}

The underlying digraph in the next construction is the circulant
$C_N^k$ used by Alon, McDiarmid, and Molloy as a sharpness example for
unlabelled cycle packings~\cite[p.~232]{AlonMcDiarmidMolloy}.

\begin{proposition}
\label{prop:circulant-obstruction}
For every $k\ge2$ and $m\ge2$, there is a strongly connected,
$k$-regular digraph on $mk$ vertices with a $\mathbb Z_k$-labelling
and no zero-sum directed cycle.
\end{proposition}

\begin{proof}
Put $N=mk$ and identify the vertices with $\mathbb Z_N$, represented by
$0,1,\ldots,N-1$.  For each vertex $x$ and each $s\in\{1,\ldots,k\}$,
include the edge
\[
  e(x,s):=\bigl(x\longrightarrow x+s\pmod N\bigr).
\]
The resulting digraph is simple, loopless, and $k$-regular.  It is
strongly connected because the edges with $s=1$ form a spanning
directed cycle.

Adapting the standard lower-bound construction for complete bidirected
graphs~\cite{MeszarosSteiner,Diwan}, label the edge with initial vertex
$x$ and step $s$ by
\[
  w(e(x,s))=
  \begin{cases}
    0,&x+s<N,\\
    1,&x+s\ge N,
  \end{cases}
  \qquad\text{in }\mathbb Z_k,
\]
where $x$ and $s$ on the right are their indicated integer
representatives.  Since $s\le k<N$, the edge wraps around the chosen
cyclic order at most once, and its label is $1$ precisely when it does so.

Let $C=x_0x_1\cdots x_{\ell-1}x_0$ be a directed cycle.  For each $i$,
let $s_i\in\{1,\ldots,k\}$ be the step used by the edge $x_ix_{i+1}$,
and let $\varepsilon_i\in\{0,1\}$ be its label, with subscripts taken
modulo $\ell$.  As an equality of integers,
\[
  x_{i+1}=x_i+s_i-N\varepsilon_i.
\]
Summing around the cycle gives
\[
  N\sum_{i=0}^{\ell-1}\varepsilon_i
  =\sum_{i=0}^{\ell-1}s_i.
\]
Write $q=\sum_i\varepsilon_i$.  Since every step is positive, $q\ge1$.
Moreover, since $C$ is simple and hence $\ell\le N$,
\[
  qN=\sum_i s_i\le k\ell\le kN.
\]
If $q=k$, equality must hold throughout, so $\ell=N$ and every step is
$k$.  This is impossible: repeatedly adding $k$ modulo $N=mk$ returns
to the initial vertex after $m<N$ steps.  Hence $1\le q\le k-1$.
The label sum of $C$ in $\mathbb Z_k$ is $q$, so it is nonzero.
\end{proof}

\section{Discussion and open problems}

As discussed in Section~1.3, minimum indegree and minimum outdegree
alone do not force a zero-sum cycle, even when both are arbitrarily
large.  The examples described there are neither strongly connected nor
Eulerian.  Diwan asked whether, for $\mathbb Z_k$-labellings, some bound
depending only on $k$ suffices under strong
connectivity~\cite{Diwan}, and for $k=2$ Thomassen proved the sharp
threshold $3$~\cite{Thomassen1992}.  In view of our results for arbitrary
finite groups, we propose the following stronger version with the
conjecturally sharp threshold, which agrees with the complete case.

\begin{conjecture}\label{conj:strongly-connected}
Let $\Gamma$ be a finite group of order $k\ge2$.  Every strongly
connected digraph with minimum indegree and minimum outdegree at least
$k+1$ contains a zero-sum directed cycle under every $\Gamma$-labelling
of its edges.
\end{conjecture}

The degree $k+1$ in Conjecture~\ref{conj:strongly-connected} cannot be
reduced in a statement uniform over all groups of order $k$:
Proposition~\ref{prop:circulant-obstruction} gives an obstruction for
$\Gamma=\mathbb Z_k$.  The conjectured sharp threshold for Eulerian
digraphs is also open.  Every weakly connected Eulerian digraph is
strongly connected: summing $d^+(v)=d^-(v)$ over a strongly connected
component with no incoming edges shows that it has no outgoing edges.
Thus Conjecture~\ref{conj:strongly-connected} would include the Eulerian
case and would replace the hypothesis of
Theorem~\ref{thm:eulerian} by the sharp condition $\delta\ge k+1$.

A weaker intermediate question is whether the exponent $2$ in the
degree-ratio condition of Theorem~\ref{thm:eulerian} can be reduced.

\begin{question}\label{question:degree-ratio}
Do there exist absolute constants $0\le\alpha<2$ and $C>0$ such that
for every finite group $\Gamma$ of order $k\ge2$, every Eulerian
digraph with minimum and maximum common degrees $\delta$ and $\Delta>0$
contains a zero-sum directed cycle under every $\Gamma$-labelling
whenever
\[
  \frac{\delta^{1+\alpha}}{\Delta^\alpha}\ge Ck?
\]
\end{question}

Theorem~\ref{thm:eulerian} is the corresponding result at $\alpha=2$.
The Eulerian case of Conjecture~\ref{conj:strongly-connected} would give
the other endpoint, $\alpha=0$.

In the complete-digraph setting, cyclic groups force the general
$|\Gamma|+1$ bound, while for $\Gamma=\mathbb Z_p^d$ the best bounds have
order $pd$ up to a logarithmic factor, rather than
$|\Gamma|=p^d$~\cite{LetzterMorrison,ChristophEtAl}.  It would be
interesting to obtain analogous existence and packing bounds here in
terms of the exponent.

\section*{Acknowledgements}

We thank Noga Alon for helpful discussions.  The author used ChatGPT
5.6 Pro for the development of the proof.  The author independently
verified all arguments and references and takes full responsibility for
the paper.

\clearpage
\bibliographystyle{amsplain}
\bibliography{zero_sum_cycles_regular_digraphs}

@misc{AkramiEtAl,
  author = {Akrami, H. and Alon, N. and Chaudhury, B. R.
            and Garg, J. and Mehlhorn, K. and Mehta, R.},
  title = {{EFX} allocations: simplifications and improvements},
  year = {2022},
  note = {arXiv:2205.07638,
          \href{https://arxiv.org/abs/2205.07638}{arXiv:2205.07638}}
}

@article{AlonBregman,
  author = {Alon, N. and Bregman, Z.},
  title = {Every $8$-uniform $8$-regular hypergraph is $2$-colorable},
  journal = {Graphs Combin.},
  volume = {4},
  number = {1},
  year = {1988},
  pages = {303--306},
  note = {\href{https://doi.org/10.1007/BF01864169}
          {doi:10.1007/BF01864169}}
}

@article{AlonDisjoint,
  author = {Alon, N.},
  title = {Disjoint directed cycles},
  journal = {J. Combin. Theory Ser. B},
  volume = {68},
  number = {2},
  year = {1996},
  pages = {167--178},
  note = {\href{https://doi.org/10.1006/jctb.1996.0062}
          {doi:10.1006/jctb.1996.0062}}
}

@article{AlonKrivelevich,
  author = {Alon, N. and Krivelevich, M.},
  title = {Divisible subdivisions},
  journal = {J. Graph Theory},
  volume = {98},
  number = {4},
  year = {2021},
  pages = {623--629},
  note = {\href{https://doi.org/10.1002/jgt.22716}
          {doi:10.1002/jgt.22716}}
}

@article{AlonLinial,
  author = {Alon, N. and Linial, N.},
  title = {Cycles of length $0$ modulo $k$ in directed graphs},
  journal = {J. Combin. Theory Ser. B},
  volume = {47},
  number = {1},
  year = {1989},
  pages = {114--119},
  note = {\href{https://doi.org/10.1016/0095-8956(89)90071-3}
          {doi:10.1016/0095-8956(89)90071-3}}
}

@article{AlonMcDiarmidMolloy,
  author = {Alon, N. and McDiarmid, C. and Molloy, M.},
  title = {Edge-disjoint cycles in regular directed graphs},
  journal = {J. Graph Theory},
  volume = {22},
  number = {3},
  year = {1996},
  pages = {231--237},
  note = {\href{https://web.math.princeton.edu/~nalon/PDFS/cycles6.pdf}
          {author's PDF}}
}

@article{BaiGrzesikLiProrok,
  author = {Bai, Y. and Grzesik, A. and Li, B. and Prorok, M.},
  title = {Cycle lengths in graphs of given minimum degree},
  journal = {J. Combin. Theory Ser. B},
  volume = {180},
  year = {2026},
  pages = {111--150},
  note = {\href{https://doi.org/10.1016/j.jctb.2026.06.003}
          {doi:10.1016/j.jctb.2026.06.003}}
}

@inproceedings{BerendsohnBoyadzhiyskaKozma,
  author = {Berendsohn, B. A. and Boyadzhiyska, S.
            and Kozma, L.},
  title = {Fixed-point cycles and approximate {EFX} allocations},
  booktitle = {47th International Symposium on Mathematical Foundations
               of Computer Science},
  series = {LIPIcs},
  volume = {241},
  year = {2022},
  pages = {17:1--17:13},
  note = {Art. 17,
          \href{https://doi.org/10.4230/LIPIcs.MFCS.2022.17}
          {doi:10.4230/LIPIcs.MFCS.2022.17}}
}

@article{BollobasModulo,
  author = {Bollob{\'a}s, B.},
  title = {Cycles modulo $k$},
  journal = {Bull. London Math. Soc.},
  volume = {9},
  number = {1},
  year = {1977},
  pages = {97--98},
  note = {\href{https://doi.org/10.1112/blms/9.1.97}
          {doi:10.1112/blms/9.1.97}}
}

@article{CampbellEtAl,
  author = {Campbell, R. and Gollin, J. P. and Hendrey, K.
            and Steiner, R.},
  title = {Optimal bounds for zero-sum cycles. {I}. {Odd} order},
  journal = {J. Combin. Theory Ser. B},
  volume = {173},
  year = {2025},
  pages = {246--256},
  note = {\href{https://doi.org/10.1016/j.jctb.2025.04.003}
          {doi:10.1016/j.jctb.2025.04.003}}
}

@article{Caro,
  author = {Caro, Y.},
  title = {Zero-sum problems---a survey},
  journal = {Discrete Math.},
  volume = {152},
  number = {1--3},
  year = {1996},
  pages = {93--113},
  note = {\href{https://doi.org/10.1016/0012-365X(94)00308-6}
          {doi:10.1016/0012-365X(94)00308-6}}
}

@article{Chu,
  author = {Chu, W.},
  title = {Determinant, permanent, and {MacMahon}'s master theorem},
  journal = {Linear Algebra Appl.},
  volume = {255},
  year = {1997},
  pages = {171--183},
  note = {\href{https://doi.org/10.1016/S0024-3795(95)00774-1}
          {doi:10.1016/S0024-3795(95)00774-1}}
}

@article{ChristophEtAl,
  author = {Christoph, M. and Knierim, C.
            and Martinsson, A. and Steiner, R.},
  title = {Improved bounds for zero-sum cycles in $\mathbb Z_p^d$},
  journal = {J. Combin. Theory Ser. B},
  volume = {173},
  year = {2025},
  pages = {365--373},
  note = {\href{https://doi.org/10.1016/j.jctb.2025.03.001}
          {doi:10.1016/j.jctb.2025.03.001}}
}

@article{ChenSaito,
  author = {Chen, G. and Saito, A.},
  title = {Graphs with a cycle of length divisible by three},
  journal = {J. Combin. Theory Ser. B},
  volume = {60},
  number = {2},
  year = {1994},
  pages = {277--292},
  note = {\href{https://doi.org/10.1006/jctb.1994.1019}
          {doi:10.1006/jctb.1994.1019}}
}

@incollection{Dean1988,
  author = {Dean, N.},
  title = {Which graphs are pancyclic modulo $k$?},
  booktitle = {Graph Theory, Combinatorics, and Applications, Vol. 1
               (Kalamazoo, MI, 1988)},
  editor = {Alavi, Y. and Chartrand, G. and Oellermann, O. R. and
            Schwenk, A. J.},
  publisher = {Wiley},
  address = {New York},
  year = {1991},
  pages = {315--326}
}

@article{DeanLesniakSaito,
  author = {Dean, N. and Lesniak, L. and Saito, A.},
  title = {Cycles of length $0$ modulo $4$ in graphs},
  journal = {Discrete Math.},
  volume = {121},
  number = {1--3},
  year = {1993},
  pages = {37--49},
  note = {\href{https://doi.org/10.1016/0012-365X(93)90535-2}
          {doi:10.1016/0012-365X(93)90535-2}}
}

@misc{Diwan,
  author = {Diwan, A. A.},
  title = {Cycles of weight divisible by $k$},
  year = {2024},
  note = {arXiv:2407.01198,
          \href{https://arxiv.org/abs/2407.01198}{arXiv:2407.01198}}
}

@article{Egorychev,
  author = {Egorychev, G. P.},
  title = {The solution of van der {Waerden}'s problem for permanents},
  journal = {Adv. Math.},
  volume = {42},
  number = {3},
  year = {1981},
  pages = {299--305},
  note = {\href{https://doi.org/10.1016/0001-8708(81)90044-X}
          {doi:10.1016/0001-8708(81)90044-X}}
}

@article{ErdosGinzburgZiv,
  author = {Erd{\H{o}}s, P. and Ginzburg, A. and Ziv, A.},
  title = {Theorem in the additive number theory},
  journal = {Bull. Res. Council Israel Sect. F},
  volume = {10F},
  year = {1961},
  pages = {41--43}
}

@inproceedings{Erdos1976,
  author = {Erd{\H{o}}s, P.},
  title = {Some recent problems and results in graph theory,
           combinatorics, and number theory},
  booktitle = {Proceedings of the Seventh Southeastern Conference on
               Combinatorics, Graph Theory and Computing},
  volume = {XVII},
  year = {1976},
  pages = {3--14},
  publisher = {Utilitas Math.}
}

@article{Falikman,
  author = {Falikman, D. I.},
  title = {Proof of the van der {Waerden} conjecture regarding the
           permanent of a doubly stochastic matrix},
  journal = {Math. Notes},
  volume = {29},
  number = {6},
  year = {1981},
  pages = {475--479},
  note = {\href{https://doi.org/10.1007/BF01163285}
          {doi:10.1007/BF01163285}}
}

@article{Friedland,
  author = {Friedland, S.},
  title = {Every $7$-regular digraph contains an even cycle},
  journal = {J. Combin. Theory Ser. B},
  volume = {46},
  number = {2},
  year = {1989},
  pages = {249--252},
  note = {\href{https://doi.org/10.1016/0095-8956(89)90047-6}
          {doi:10.1016/0095-8956(89)90047-6}}
}

@article{GoldsteinGraham,
  author = {Goldstein, A. J. and Graham, R. L.},
  title = {A {Hadamard}-type bound on the coefficients of a determinant
           of polynomials},
  journal = {SIAM Rev.},
  volume = {15},
  number = {3},
  year = {1973},
  pages = {657--658},
  note = {\href{https://doi.org/10.1137/1015079}
          {doi:10.1137/1015079}}
}

@article{GaoHuoLiuMa,
  author = {Gao, J. and Huo, Q. and Liu, C.-H. and Ma, J.},
  title = {A unified proof of conjectures on cycle lengths in graphs},
  journal = {Int. Math. Res. Not. IMRN},
  volume = {2022},
  number = {10},
  year = {2022},
  pages = {7615--7653},
  note = {\href{https://doi.org/10.1093/imrn/rnaa324}
          {doi:10.1093/imrn/rnaa324}}
}

@article{GutinSudakovYeo,
  author = {Gutin, G. and Sudakov, B. and Yeo, A.},
  title = {Note on alternating directed cycles},
  journal = {Discrete Math.},
  volume = {191},
  number = {1--3},
  year = {1998},
  pages = {101--107},
  note = {\href{https://doi.org/10.1016/S0012-365X(98)00097-1}
          {doi:10.1016/S0012-365X(98)00097-1}}
}

@misc{GurvitsBethe,
  author = {Gurvits, L.},
  title = {Unleashing the power of {Schrijver}'s permanental inequality
           with the help of the {Bethe} approximation},
  year = {2011},
  note = {arXiv:1106.2844,
          \href{https://arxiv.org/abs/1106.2844}{arXiv:1106.2844}}
}

@misc{KasselLevy,
  author = {Kassel, A. and L{\'e}vy, T.},
  title = {Trace identities for quiver representations},
  year = {2026},
  note = {arXiv:2603.22156,
          \href{https://arxiv.org/abs/2603.22156}{arXiv:2603.22156}}
}

@article{KellyKuhnOsthus,
  author = {Kelly, L. and K{\"u}hn, D. and Osthus, D.},
  title = {Cycles of given length in oriented graphs},
  journal = {J. Combin. Theory Ser. B},
  volume = {100},
  number = {3},
  year = {2010},
  pages = {251--264},
  note = {\href{https://doi.org/10.1016/j.jctb.2009.08.002}
          {doi:10.1016/j.jctb.2009.08.002}}
}

@article{LetzterMorrison,
  author = {Letzter, S. and Morrison, N.},
  title = {Directed cycles with zero weight in $\mathbb Z_p^k$},
  journal = {J. Combin. Theory Ser. B},
  volume = {168},
  year = {2024},
  pages = {192--207},
  note = {\href{https://doi.org/10.1016/j.jctb.2024.05.002}
          {doi:10.1016/j.jctb.2024.05.002}}
}

@article{Lossers,
  author = {Lossers, O. P.},
  title = {A {Hadamard}-type bound on the coefficients of a determinant
           of polynomials ({A}. {J}. {Goldstein} and {R}. {L}. {Graham})},
  journal = {SIAM Rev.},
  volume = {16},
  number = {3},
  year = {1974},
  pages = {394--395},
  note = {\href{https://doi.org/10.1137/1016065}
          {doi:10.1137/1016065}}
}

@misc{LuoMaZhao,
  author = {Luo, Y. and Ma, J. and Zhao, Z.},
  title = {Dean's conjecture and cycles modulo $k$},
  year = {2026},
  note = {arXiv:2601.13552,
          \href{https://arxiv.org/abs/2601.13552}{arXiv:2601.13552}}
}

@book{MacMahon,
  author = {MacMahon, P. A.},
  title = {Combinatory Analysis},
  volume = {1--2},
  publisher = {Cambridge University Press},
  address = {Cambridge},
  year = {1915--1916}
}

@article{MeszarosSteiner,
  author = {M{\'e}sz{\'a}ros, T. and Steiner, R.},
  title = {Zero-sum cycles in complete digraphs},
  journal = {European J. Combin.},
  volume = {98},
  year = {2021},
  pages = {Art. 103399},
  note = {\href{https://doi.org/10.1016/j.ejc.2021.103399}
          {doi:10.1016/j.ejc.2021.103399}}
}

@article{SudakovVerstraete,
  author = {Sudakov, B. and Verstra{\"e}te, J.},
  title = {The extremal function for cycles of length $\ell$ mod $k$},
  journal = {Electron. J. Combin.},
  volume = {24},
  number = {1},
  year = {2017},
  pages = {Paper No. 1.7},
  note = {\href{https://doi.org/10.37236/6257}
          {doi:10.37236/6257}}
}

@article{Thomassen1985,
  author = {Thomassen, C.},
  title = {Even cycles in directed graphs},
  journal = {European J. Combin.},
  volume = {6},
  number = {1},
  year = {1985},
  pages = {85--89},
  note = {\href{https://doi.org/10.1016/S0195-6698(85)80025-1}
          {doi:10.1016/S0195-6698(85)80025-1}}
}

@article{ThomassenModulo,
  author = {Thomassen, C.},
  title = {Graph decomposition with applications to subdivisions and
           path systems modulo $k$},
  journal = {J. Graph Theory},
  volume = {7},
  number = {2},
  year = {1983},
  pages = {261--271},
  note = {\href{https://doi.org/10.1002/jgt.3190070215}
          {doi:10.1002/jgt.3190070215}}
}

@article{Thomassen1992,
  author = {Thomassen, C.},
  title = {The even cycle problem for directed graphs},
  journal = {J. Amer. Math. Soc.},
  volume = {5},
  number = {2},
  year = {1992},
  pages = {217--229},
  note = {\href{https://doi.org/10.1090/S0894-0347-1992-1135027-1}
          {doi:10.1090/S0894-0347-1992-1135027-1}}
}

@article{VaziraniYannakakis,
  author = {Vazirani, V. V. and Yannakakis, M.},
  title = {Pfaffian orientations, $0$--$1$ permanents, and even cycles
           in directed graphs},
  journal = {Discrete Appl. Math.},
  volume = {25},
  number = {1--2},
  year = {1989},
  pages = {179--190},
  note = {\href{https://doi.org/10.1016/0166-218X(89)90053-X}
          {doi:10.1016/0166-218X(89)90053-X}}
}

\end{document}